\documentclass[10pt]{article}
\usepackage{amsmath,amsthm,amsfonts,amssymb,url}
\usepackage{color}
\usepackage{BICA}

\newtheorem{theorem}{Theorem}[section]
\newtheorem{lemma}[theorem]{Lemma}
\newtheorem{corollary}[theorem]{Corollary}
\newtheorem{example}{Example}[section]

\newcommand{\zed}{\ensuremath{\mathbb{Z}}} 

\title{Designing Progressive Dinner Parties}
\author{Douglas R.\ Stinson\thanks{The author's research is supported by  NSERC discovery grant RGPIN-03882.
}}
\affil{David R.\ Cheriton School of Computer Science\\ University of Waterloo\\
Waterloo, Ontario, N2L 3G1, Canada}

\date{}

\begin{document}
\maketitle

\begin{abstract}
I recently came across a combinatorial design problem involving progressive dinner parties (also known as  safari suppers). In this note, I provide some elementary methods of designing schedules for these kinds of 
dinner parties.
\end{abstract}

\section{The Problem}
\label{intro.sec}

A simple form of \emph{progressive dinner party} could involve three couples eating a three-course dinner, with each couple hosting one course. 
I received email from Julian Regan asking if there was a nice way to design a more complicated type of progressive dinner party, which he described as follows:
\begin{quote}The event involves a number of couples having each course of a three-course meal at a different person's house, with three couples at each course, every couple hosting once and no two couples meeting more than once.
\end{quote}
Let us represent each couple by a \emph{point} $x \in X$ and each course of each meal by a \emph{block} consisting of three points. Suppose there are $v$ points (i.e., couples). Evidently we want a collection of blocks of size three, say $\mathcal{B}$,
such that the following conditions are satisfied:
\begin{enumerate}
\item The blocks can be partitioned into three parallel classes, each consisting of $v/3$ disjoint blocks. (Each parallel class corresponds to a specific course of the meal.)
Hence, there are a total of $v$ blocks and we require $v \equiv 0 \bmod 3$.
\item No pair of points occurs in more than one block.
\item There is a bijection $h : \mathcal{B} \rightarrow X$ such that  
$h(B) \in B$ for all $B \in \mathcal{B}$. (That is, we can identify a \emph{host} for each block in such a way that each point occurs as a host exactly once.)
\end{enumerate}
We will refer to such a collection of blocks as a PDP$(v)$.

It is not hard to see that a PDP$(v)$ does not exist if $v=3$ or $v=6$, because we cannot satisfy condition 2. However, for all larger values of $v$ divisible by three, we show in Section \ref{solns.sec} that it is possible to construct a PDP$(v)$. Section \ref{generalization.sec} considers a generalization of the problem in which there are $k$ courses and $k$ couples present at each course, and gives a complete solution when $k = 4$ or $k=5$.

\section{Two Solutions}
\label{solns.sec}

We begin with a simple construction based on latin squares. A \emph{latin square} of order $n$ is an $n$ by $n$ array of $n$ symbols, such that each symbol occurs in exactly one cell in each row and each column of the array. A \emph{transversal} of a latin square of order $n$ is a set of $n$ cells, one from each row and each column, that contain $n$ different symbols. Two transversals are \emph{disjoint} if they do not contain any common cells.

\begin{lemma}
\label{LS.lem}
Suppose there is a latin square of order $w$ that contains three  disjoint transversals. Then there is a PDP$(3w)$. 
\end{lemma}

\begin{proof}
Let $L$ be a latin square of order $w$ that contains disjoint transversals $T_1, T_2$ and $T_3$.
Let the rows of $L$ be indexed by $R$, let the columns be indexed by $C$ and let the symbols be indexed by $S$. We assume that $R$, $C$ and $S$ are three mutually disjoint sets. Each transversal $T_i$  
consists of $w$ ordered pairs in $R \times C$. 

We will construct a PDP$(3w)$ on points $X = R \cup C \cup S$.  For $1 \leq i \leq 3$, we construct a parallel class $P_i$ as follows:
\[ P_i = \{ \{r,c,L(r,c)\} : (r,c) \in T_i \}.\]
Finally, for any block $B = \{r,c,s\}\in P_1 \cup P_2 \cup P_3$, we define $h(B)$ as follows:
\begin{itemize}
\item if $B  \in P_1$, then $h(B) = r$
\item if $B  \in P_2$, then $h(B) = c$
\item if $B  \in P_3$, then $h(B) = s$.
\end{itemize}
The verifications are straightforward. 
\begin{itemize}
\item
First, because each $T_i$ is a transversal, it is clear that each $P_i$ is a parallel class. 
\item
No pair of points $\{r,c\}$ occurs in more than one block because the three transversals are disjoint. 
\item
Suppose a pair of points $\{r,s\}$ occurs in more than one block. Then
there is $L(r,c) \in T_i$ and $L(r,c') \in T_j$ such that $L(r,c) = L(r,c')$. $T_i$ and $T_j$ are disjoint, so $c \neq c'$. But then we have two occurrences of the same symbol in row $r$ of $L$, which contradicts the assumption that $L$ is a latin square.  
\item
The argument that no pair of points $\{c,s\}$ occurs in more than one block is similar.
\item
Finally, the mapping $h$ satisfies property 3 because each $T_i$ is a transversal.
\end{itemize}
\end{proof}

\begin{theorem}
There
is a PDP$(3w)$ for all $w \geq 3$. 
\end{theorem}

\begin{proof}
If $\geq 3$, $w \neq 6$, there is a pair of orthogonal latin squares of order $w$. It is well-known that a pair of orthogonal latin squares of order $w$ is equivalent to a latin square of order $w$ that contains $w$ disjoint transversals (see, e.g., \cite[p.\ 162]{HCD}). Since $w \geq 3$, we have three disjoint transversals and we can apply Lemma \ref{LS.lem} to obtain a PDP$(w)$. There do not exist a pair of orthogonal latin squares of order 6, but there is a latin square of order $6$ that contains four disjoint transversals (see, e.g., \cite[p.\ 193]{HCD}). So we can also use Lemma \ref{LS.lem}  to construct a PDP$(18)$.
\end{proof}

\begin{example}
\label{12.ex} We construct a PDP$(12)$. Start with a pair of orthogonal latin squares of order $4$:
\[
\begin{array}{rr}
L_1 = 
\begin{array}{|c|c|c|c|}
\hline 
1 & 3 & 4 & 2\\ \hline
4 & 2 & 1 & 3\\ \hline
2 & 4 & 3 & 1\\ \hline
3 & 1 & 2 & 4\\ \hline
\end{array}\, ,
&
L_2 = 
\begin{array}{|c|c|c|c|}
\hline 
1 & 4 & 2 & 3\\ \hline
3 & 2 & 4 & 1\\ \hline
4 & 1 & 3 & 2\\ \hline
2 & 3 & 1 & 4\\ \hline
\end{array}.
\end{array}
\]
Each symbol in $L_2$ gives us a transversal in $L_1$. Suppose we index the rows by $r_i$ ($1 \leq i \leq 4$) and the columns by $c_j$ ($1 \leq j \leq 4$). From symbols $1,2$ and $3$, we obtain the following three disjoint transversals in $L_1$:
\begin{align*}
T_1 &= \{ (r_1,c_1), (r_2,c_4), (r_3,c_2), (r_4,c_3)\}\\
T_2 &= \{ (r_1,c_3), (r_2,c_2), (r_3,c_4), (r_4,c_1)\}\\
T_3 &= \{ (r_1,c_4), (r_2,c_1), (r_3,c_3), (r_4,c_2)\}.
\end{align*}
Suppose we relabel the points as $1, \dots 12$, replacing $r_1,\dots , r_4$ by $1, \dots , 4$;
replacing $c_1,\dots , c_4$ by $5, \dots , 8$;
and replacing the symbols $1,\dots , 4$ by $9, \dots , 12$.
Then
we obtain the following PDP$(12)$, where the hosts are indicated in red:
\begin{align*}
P_1 &= \{ 
\{\textcolor{red}{1},5,9\},
\{\textcolor{red}{2},8,11\}, 
\{\textcolor{red}{3},6,12\}, 
\{\textcolor{red}{4},7,10\}
\} \\
P_2 &= \{ 
\{1,\textcolor{red}{7},12\},
\{2,\textcolor{red}{6},10\} ,
\{3,\textcolor{red}{8}, 9\},
\{4,\textcolor{red}{5},11\} 
\} \\
P_3 &= \{ 
\{1,8,\textcolor{red}{10}\} ,
\{2,5,\textcolor{red}{12}\} ,
\{3,7, \textcolor{red}{11}\} ,
\{4,6,\textcolor{red}{9}\} 
\}.
\end{align*}
\end{example}

\bigskip

Of course, using a pair of latin squares is overkill. It would perhaps be easier just to give explicit formulas to construct a PDP. Here is one simple solution that works for all $v \geq 9$ such that $v \equiv 0 \bmod 3$ and $v \neq 12$.

\begin{theorem}
\label{direct.thm}
Let $w \geq 3$, $w \neq 4$, and let $X = \zed_w \times \{0,1,2\}$. Define the following three parallel classes:
\begin{align*}
P_0 &= \{ \{(0,0), (0,1), (0,2)\} \bmod w\}\\
P_1 &= \{ \{(0,0), (1,1), (2,2)\} \bmod w\}\\
P_2 &= \{ \{(0,0), (2,1), (4,2)\} \bmod w\}.
\end{align*}
For any block $B = \{(i,0), (j,1), (k,2)\} \in P_0 \cup P_1 \cup P_2$, define  $h(B)$ as follows. 
\begin{itemize}
\item if $B  \in P_0$, then $h(B) = (i,0)$
\item if $B  \in P_1$, then $h(B) = (j,1)$
\item if $B  \in P_2$, then $h(B) = (k,2)$.
\end{itemize} Then $P_0, P_1, P_2,$ and $h$ yield a PDP$(3w)$.
\end{theorem}
\begin{proof}
It is clear that each $P_i$ is a parallel class because we are developing a base block modulo $w$ and each base block contains one point with each possible second coordinate. For the same reason, the mapping $h$ satisfies property 3.

Consider the differences $(y -x) \bmod w$ that occur between  pairs of 
points $\{ (x,0), (y,1)\}$. We obtain all pairs with differences $0,1$ and $2$ when we develop the three base blocks. The same thing happens when we look at the differences $(y -x) \bmod w$ between pairs of 
points $\{ (x,1), (y,2)\}$.
 
Finally, consider the differences $(y -x) \bmod w$ that occur between  pairs of 
points $\{ (x,0), (y,2)\}$. We obtain all pairs with differences $0,2$ and $4$ modulo $w$ when we develop the three base blocks. Since $w \neq 4$, these differences are distinct and the pairs obtained by developing the base blocks are also distinct. 
\end{proof}

If $w = 4$, then the construction given in Theorem \ref{direct.thm}  does not yield a PDP$(12)$,  because various pairs occur in more than one block. For example, the pair $\{(0,0), (0,2)\}$ occurs in a block of $P_0$ as well as in a block of $P_2$. 

\begin{example}
We apply Theorem \ref{direct.thm} with $w=5$. The three parallel classes, with hosts in red, are:
\[
\begin{array}{c|c|c}
P_0 & P_1 & P_2 \\ \hline 
\{ \textcolor{red}{(0,0)}, (0,1), (0,2) \} & \{ (0,0), \textcolor{red}{(1,1)}, (2,2) \} & \{ (0,0), (2,1), \textcolor{red}{(4,2)} \} \\
\{ \textcolor{red}{(1,0)}, (1,1), (1,2) \} & \{ (1,0), \textcolor{red}{(2,1)}, (3,2) \} & \{ (1,0), (3,1), \textcolor{red}{(0,2)} \} \\
\{ \textcolor{red}{(2,0)}, (2,1), (2,2) \} & \{ (2,0), \textcolor{red}{(3,1)}, (4,2) \} & \{ (2,0), (4,1), \textcolor{red}{(1,2)} \} \\
\{ \textcolor{red}{(3,0)}, (3,1), (3,2) \} & \{ (3,0), \textcolor{red}{(4,1)}, (0,2) \} & \{ (3,0), (0,1), \textcolor{red}{(2,2)} \} \\
\{ \textcolor{red}{(4,0)}, (4,1), (4,2) \} & \{ (4,0), \textcolor{red}{(0,1)}, (1,2) \} & \{ (4,0), (1,1), \textcolor{red}{(3,2)} \} \end{array}
\]
\end{example}

\subsection{Finding Hosts}

The specific constructions that we provided in Section \ref{solns.sec} led to a very simple method to identify hosts. However, no matter what collection of three parallel classes we use, it will be possible
to define hosts in such a way that property 3 of a PDP will be satisfied.

\begin{theorem}
Suppose that $P_1,P_2$ and $P_3$ are three parallel classes of blocks of size three, containing points from a set $X$ of size $v \equiv 0 \bmod 3$. 
Then we can define a mapping $h$ that satisfies property 3. 
\end{theorem}

\begin{proof} 
Construct the bipartite 
point-block incidence graph of the design. The nodes in this graph are all the elements of $X \cup \mathcal{B}$. For $x \in X$ and $B\in \mathcal{B}$, we create an edge from $x$ to $B$ if and only if $x \in B$. The resulting graph is a 3-regular bipartite graph and hence it has a perfect matching $M$
(this is a corollary of Hall's Theorem, e.g., see \cite[Corollary 16.6]{BM}). For every $B \in \mathcal{B}$, define $h(B) = x$, where $x$ is the point matched with $B$ in the matching $M$. 
\end{proof}

The following corollary is immediate.

\begin{corollary}
Suppose that $P_1,P_2$ and $P_3$ are three parallel classes of blocks of size three, containing points from a set $X$ of size $v \equiv 0 \bmod 3$. 
Suppose also that no pair of points occurs in more one block in $\mathcal{B} = P_1 \cup P_2 \cup P_3$. 
Then there is a PDP$(v)$.
\end{corollary}

\section{A Generalization}
\label{generalization.sec}

Suppose we now consider a generalization where meals have $k$ courses and each course includes $k$ couples. We define a PDP$(k,v)$ to be a set of blocks of size $k$, defined on a set of $v$ points, which satisfies the following properties:
\begin{enumerate}
\item The blocks can be partitioned into $k$ parallel classes, each consisting of $v/k$ disjoint blocks. Hence, there are a total of $v$ blocks and we require $v \equiv 0 \bmod k$.
\item No pair of points occurs in more than one block.
\item There is a bijection $h : \mathcal{B} \rightarrow X$ such that  
$h(B) \in B$ for all $B \in \mathcal{B}$. 
\end{enumerate}
The problem we considered in Section \ref{intro.sec} was just the special case $k=3$ of this general definition.

Here is a simple necessary condition for existence of a PDP$(k,v)$.
\begin{lemma}
If a  PDP$(k,v)$ exists, then $v \geq k^2$.
\end{lemma}

\begin{proof} 
A given point $x$ occurs in $k$ blocks, each having size $k$. The points in these blocks (excluding $x$) must be distinct.
Therefore, \[v \geq k(k-1)+ 1 = k^2 - (k-1).\] Since $k$ divides $v$, we must have $v \geq k^2$.
\end{proof}

We have the following results that are straightforward generalizations of our  results from Section 
\ref{solns.sec}. The first three of these results are stated without proof.

\begin{lemma}
\label{LS-k.lem}
Suppose there are $k-2$ orthogonal latin squares of order $w$ that contain $k$ disjoint common transversals. Then there is a PDP$(k,kw)$. 
\end{lemma}

\begin{corollary}
\label{LS-k.cor}
Suppose there are $k-1$ orthogonal latin squares of order $w$. Then there is a PDP$(k,kw)$. 
\end{corollary}

\begin{theorem}
\label{hosts-k.thm}
Suppose that $P_1, \dots , P_k$ are $k$ parallel classes of blocks of size $k$, containing points from a set $X$ of size $v \equiv 0 \bmod k$. 
Then we can define a mapping $h$ that satisfies property 3. 
\end{theorem}

Our last construction generalizes Theorem \ref{direct.thm}.

\begin{theorem}
\label{direct-k.thm}
Let $w \geq k \geq 3$. Suppose that the following condition holds:
\begin{equation}
\label{num.eq}
\text{There is no factorization $w = st$ with
$2 \leq s \leq k-1$ and $2 \leq t \leq k-1$.}
\end{equation}
Then there is a PDP$(k,kw)$.
\end{theorem}

\begin{proof}
Define $X = \zed_w \times \{0,\dots , k-1\}$ and define the following $k$ parallel classes, $P_0, \dots , P_{k-1}$:
\begin{align*}
P_i &= \{ \{(0,0), (i,1), (2i,2), \dots , ((k-1)i, k-1)\} \bmod w\},
\end{align*}
for $i = 0, \dots , k-1$.
Finally, define the mapping $h$ as follows. For any block $B \in P_{\ell}$, define 
$h(B) = (x,\ell)$, where $(x,\ell)$ is the unique point in $B$ having second coordinate equal to $\ell$.
Then $P_0, \dots , P_{k-1}$ and $h$ yield a PDP$(k,kw)$.

Most of the verifications are straightforward, but it would perhaps be useful to see how condition  
(\ref{num.eq}) arises. Consider the differences $(y -x) \bmod w$ that occur between  pairs of points $\{ (x,c), (y,c+d)\}$, where $c$ and $d$ are fixed,
$0 \leq c \leq k-2$, $1 \leq d \leq k-c-1$. These difference are
\[0,d,2d, \dots ,(k-1)d \bmod w,\]
where $0 < d \leq k-1$.
We want all of these differences to be distinct.
Suppose that 
\[ id \equiv jd \bmod w\]
where $0 \leq j < i \leq k-1$.
Then \[ (i-j)d \equiv 0 \bmod w.\]
Hence, \[ ed \equiv 0 \bmod w\] where $0 < e \leq k-1$ and $0 < d \leq k-1$.
Then, it not hard to see that $w$ can be factored as the product of two positive integers, both of which are 
at most $k-1$.

Conversely, suppose such a factorization exists, say $w = st$. Then  the pair
$\{ (0,0) , (0,t)\}$ occurs in a block in $P_0$ and again in a block in $P_s$.
\end{proof}

Observe that  condition (\ref{num.eq})  of Theorem \ref{direct-k.thm} holds 
if $w$ is prime or if $w > (k-1)^2$. Therefore we have the following corollary of Theorem \ref{direct-k.thm}.

\begin{corollary}
\label{direct-k.cor}
Let $w \geq k \geq 3$. Suppose  that $w$ is prime or $w > (k-1)^2$. Then there is a PDP$(k,kw)$.
\end{corollary}

In general, some values of $w$ will be ruled out (in the sense that Theorem \ref{direct-k.thm} cannot be applied) for a given value of $k$. For example, as we have already seen in the previous section, we cannot take $w=4$ in Theorem \ref{direct-k.thm} if $k = 3$. However, a PDP$(12)$ was constructed by a different method in Example \ref{12.ex}. 

We have the following  complete results for $k=4$ and $k=5$.

\begin{theorem}
\label{k=4,5.thm} There is a PDP$(4,4w)$ if and only if $w \geq 4$. 
Further, there is a PDP$(5,5w)$ if and only if $w \geq 5$. 
\end{theorem}

\begin{proof}
For $k=4$, we proceed as follows. Theorem \ref{direct-k.thm} yields a PDP$(4,4w)$ for all $w \geq 4$, $w\neq 4, 6, 9$.
Theorem \ref{LS-k.cor} provides a PDP$(4,16)$ and a PDP$(4,36)$ since three orthogonal latin squares of orders $4$ and $9$ are known to exist (see \cite{HCD}). The last case to consider is $w=6$. Here we can use a resolvable $4$-GDD of type $3^8$ (\cite{Shen}). Actually, we only need four of the seven parallel classes in this design. Then, to define the hosts, we can use Theorem \ref{hosts-k.thm}.

We handle $k=5$ in a similar manner. Theorem \ref{direct-k.thm} yields a PDP$(5,5w)$ for all $w \geq 5$, $w\neq 6,8,9, 12$ or $16$. There are four orthogonal latin squares of orders $8,9, 12$ and $16$ (see \cite{HCD}) so these values of $w$ are taken care of by Theorem \ref{LS-k.cor}. 

Finally, the value $w=6$ is handled by a direct construction due to Marco Buratti \cite{marco}.
Define $X= \zed_{30}$ and \[\mathcal{B} = \{ \{ 0,1,8,12,14\} \bmod 30 \}.\] So we have thirty blocks that are obtained from the base block $B_0 = \{ 0,1,8,12,14\}$. It is easy to check that no pair of points is repeated, because the differences of pairs of points occurring in $B_0$ 
are all those in the set \[\pm \{1, 2, 4,6,7,8,11,12,13,14\}.\]

Define \[P_0=\{B_0+5j \bmod 30 : j=0,1,\dots,5\}\] and for $1 \leq i \leq 4$, let
\[P_i= \{ B+i \bmod 30: B \in P_0 \}.\]
In this way, $\mathcal{B}$ is partitioned into five parallel classes, each containing six blocks. 

Theorem \ref{hosts-k.thm} guarantees that we can define hosts in a suitable fashion. However, it is easy to write down an explicit formula, namely, 
$ h(B_0 + i) = i$ for $0 \leq i \leq 29$.
\end{proof}

\section*{Acknowledgements}

I would like to thank Julian Regan for bringing this problem to my attention. Thanks also to Marco Buratti for providing the construction for the PDP$(5,30)$ that was used in the proof of Theorem \ref{k=4,5.thm}.

\end{document}